\documentclass[12pt]{amsart}
\usepackage{amssymb}
\usepackage{amsmath}
\usepackage{amsthm} 
\usepackage{graphicx,subcaption}
\usepackage{stmaryrd}
\usepackage{physics}
\usepackage[utf8]{inputenc}
\usepackage{amsfonts}
\usepackage{graphicx}
\usepackage{hyperref}
\usepackage{xcolor} 

\usepackage[foot]{amsaddr}

\makeatletter
\renewcommand{\email}[2][]{
  \ifx\emails\@empty\relax\else{\g@addto@macro\emails{,\space}}\fi%
  \@ifnotempty{#1}{\g@addto@macro\emails{\textrm{(#1)}\space}}%
  \g@addto@macro\emails{#2}
}
\makeatother

\hypersetup{
	colorlinks   = true, 
	urlcolor     = blue, 
	linkcolor    = blue, 
	citecolor   = blue 
}

\numberwithin{equation}{section}
\newtheorem{theorem}{Theorem}[section]
\newtheorem{ex}[theorem]{Example}
\newtheorem{lemma}[theorem]{Lemma}
\newtheorem*{theorem*}{Theorem}
\newtheorem*{lemma*}{Lemma}
\newtheorem{claim}[theorem]{Claim}

 \newtheorem{remark}[theorem]{Remark}\newtheorem{proposition}[theorem]{Proposition}\newtheorem{definition}[theorem]{Definition}
\newtheorem{corollary}[theorem]{Corollary} \newcommand{\La}{\Lambda}
\newcommand{\R}{{\mathbb R}}  \newcommand{\Z}{{\mathbb Z}} \newcommand{\N}{{\mathbb N}}

\newcommand{\Cc}{{\mathbb C}}   \newcommand{\cc}{{\bf c}}
\newcommand{\Lam}{\Lambda}
 \newcommand{\Kk}{\mathcal{K}} \newcommand{\ccc}{\mathcal{S_\delta}}

\newcommand{\re}{{\rm Re}}\newcommand{\Ff}{\mathcal{G}}\newcommand{\Rr}{\mathcal{R}}\newcommand{\kk}{\mathcal{C}}\newcommand{\Hh}{\mathcal{H}}

\title[On  Gabor frames generated by ratios of exponential polynomials]{Sampling  in quasi shift-invariant spaces and  Gabor frames generated by ratios of exponential polynomials}
\thanks{I. Z. gratefully acknowledge support from the Austrian Science Fund (FWF) [\href{https://doi.org/10.55776/P33217}{10.55776/P33217}] and from the Research Council of Norway by Grant 334466, “Fourier Methods
and Multiplicative Analysis”}

 \author{Alexander Ulanovskii$^*$}
 \address{$*$ Department of Mathematics and Physics, University of Stavanger, 4036 Stavanger, Norway}
 \email{alexander.ulanovskii@uis.no}
 
 \author{Ilya Zlotnikov$^{\circ,\lhd}$}
 \address{$\circ$ Faculty of Mathematics, University of Vienna, Oskar-Morgenstern-Platz 1,
		A-1090 Vienna, Austria
  }\address{
  $\lhd$ Department of Mathematical Sciences, Norwegian University of Science and Technology (NTNU), 7491 Trondheim, Norway
  }
\email{ilia.zlotnikov@univie.ac.at}


\begin{document}
\begin{abstract}
We introduce two families of generators (functions) $\Ff$ that consist of entire and meromorphic functions enjoying a certain periodicity property and contain the classical Gaussian and hyperbolic secant generators. Sharp results are proved on the density of separated sets that provide non-uniform sampling for the shift-invariant and quasi-shift-invariant spaces generated by elements of these families.
As an application, new sharp results are obtained on the density of semi-regular lattices for the Gabor frames generated by elements from these families.

\end{abstract}
 \date{}\maketitle

\section{Introduction and main results}

A countable set $\Gamma\subset\R$ is called separated if 
$$\inf_{\gamma,\gamma'\in\Gamma, \gamma \neq \gamma'}|\gamma-\gamma'|>0.$$
Given  a  generator (function) $\Ff$ with a "reasonably fast" decay at $\pm\infty$ and a number  $p,1\leq p\leq\infty$, the shift-invariant space $V_\Z^p(\Ff)$ consists of all functions $f$ of the form 
$$f(x)=\sum_{n\in\Z}c_n \Ff(x-n),\quad \{c_n\}\in l^p(\Z).$$ 

More generally, given a separated set $\Gamma\subset\R$,  the quasi shift-invariant space $V_\Gamma^p(\Ff)$  consists of all functions of the form 
$$f(x)=\sum_{\gamma\in\Gamma}c_\gamma \Ff(x-\gamma),\quad \{c_\gamma\}\in l^p(\Gamma).$$

 An important class of generators is the Wiener amalgam space $W_0$, which consists of  measurable  functions $\Ff: \R \to \Cc,$ satisfying \begin{equation}\label{wiener}\|\Ff\|_W := \sum\limits_{k \in \Z} \|\Ff\|_{L^\infty (k,k+1)} < \infty.\end{equation}

Shift-invariant and quasi shift-invariant  spaces have important applications in mathematics and engineering, in particular since they are often used as models for spaces of signals and images. It is also well known that there is a close connection between the Gabor frames and sampling sets for the shift-invariant spaces.

A classical example is the Paley--Wiener space $PW^2_{\pi}$ which is exactly the shift-invariant space $V^2_\Z(\Ff)$ generated by the sinc function $\Ff(x)=\sin(\pi x)/(\pi x).$
The remarkable result in digital signal processing is the  Shannon--Whittaker--Kotelnikov sampling theorem that states that every $f\in PW^2_{\pi}$ can be reconstructed from its values at the integers:
$$f(x)=\sum_{n\in\Z}f(n)\,\mbox{sinc}(x-n).$$
This implies that the set of integers $\Z$ is a stable sampling set for $PW^2_{\pi}$.

The theory of shift-invariant spaces is by now well developed and a number of sampling theorems are proved for various generators. Due to the  mentioned relation between Gabor frames and sampling in shift-invariant spaces,  certain sampling theorems are available for 
\begin{itemize}
    \item  B-splines, see \cite{MR1756138, gröchenig2023sampling},
    \item  Hermite functions, see  \cite{MR2292280,MR2529475, MR3027914},
    \item Truncated and symmetric exponential functions, see \cite{MR1621312,MR1964306}, 
    \item  Gaussian kernel (see \cite{Luef_gabor, MR1188007, PR_gabor, MR4782146, MR1173117, MR1173118}),  hyperbolic secant (see \cite{MR1884237} and a very recent paper \cite{baranov2023irregular}), and, more generally, totally positive functions, see \cite{grs, MR3053565},
    \item  Rational functions, see \cite{MR4345947, MR4542702}.
\end{itemize}
 
The sampling problem for quasi shift-invariant spaces is significantly more complicated and results are scarcer.
For the {\it totally positive} generators $\Ff$ of {\it finite type}, a sufficient condition for stable sampling in terms of covering (or maximum gap\footnote{See also a more general condition $(C_r(\varepsilon))$ in Theorem~16 in \cite{MR3053565}.}) for a quasi shift-invariant space was obtained in \cite[Theorems 2 and 16]{MR3053565}. 
The approach in this paper was based on an  
application of Schoenberg and Whitney's characterization of the invertibility of a pre-Gramian matrix generated by a totally positive function of finite type. 

In this paper, we introduce two families of generators. Using complex-analytic methods, we prove sharp results on the density of sampling sets for the corresponding shift-invariant and quasi shift-invariant spaces. As an application, we obtain new sharp results on the density of semi-regular lattices for the Gabor frames with generators from these families.

\subsection{Sampling sets and Beurling densities}

 A separated set $\Lambda\subset\R$ is called  a (stable) sampling set for $V^p_\Gamma(\Ff)$ if   the following sampling inequalities  $$A\|f\|_p^p\leq \sum_{\lambda\in\Lambda}|f(\lambda)|^p\leq B\|f\|_p^p, \quad 1 \le p < \infty,$$ $$\|f\|_{\infty} \leq K \sup\limits_{\lambda \in \Lambda} |f(\lambda)|,\quad p=\infty,$$ hold true with some positive constants $A,B,K$ and for every $ f\in V_\Gamma^p(\Ff).$ 

The  {\it lower and upper uniform densities}  of a separated set $\Lambda$ (sometimes called the Beurling densities) are defined by 
\begin{equation}\label{lowed}
   D^-(\Lambda) := \lim_{R\to\infty} \inf_{x\in \mathbb{R}} \frac{\# (\Lambda \cap [x-R,x+R])}{2R},
\end{equation}
\begin{equation}
   D^+(\Lam) := \lim_{R\to\infty} \sup_{x\in \mathbb{R}} \frac{\# (\Lam \cap [x-R,x+R])}{2R}. 
\end{equation}
These densities play a key role in the study of sampling and interpolation sets.

Throughout the paper, $\Gamma$ denotes a separated set of translates.
 To avoid trivial remarks, in what follows we always assume that  $\Gamma$ is {\it relatively dense}, i.e.  $D^-(\Gamma)>0.$

We will now introduce two classes of generators.

\subsection{Class \texorpdfstring{$\Kk(\alpha)$}{K(a)}}\label{sub_s_g}
Given a number $\alpha > 0$  and a rational function $\mathcal{R}=P/Q$,
we consider $2\pi i/\alpha$-periodic generator
\begin{equation}\label{g}\Ff(z)=\mathcal{R}(e^{\alpha z})=\frac{P(e^{\alpha z})}{Q(e^{\alpha z})}.\end{equation}

\begin{definition}\label{admis_def} 
We denote by $\Kk(\alpha)$ the class of all generators $\Ff$ defined in \eqref{g} where $P$ and $Q$ are non-trivial complex polynomials without common zeros and satisfying the following three conditions:
\begin{enumerate}
    \item[\rm ($A$)] $1\leq {\rm deg} \,P< {\rm deg}\, Q${\rm;}
    \item[\rm ($B$)] $P(0)=0${\rm;}
    \item[\rm ($C$)] $Q(x)\ne0, x\geq0$.
\end{enumerate}
\end{definition}

One may  check that conditions $(A)-(C)$ above are  necessary and sufficient for the  generator $\Ff$ defined in \eqref{g} to be integrable on $\R$,  and that these conditions imply
\begin{equation}\label{g0}|\Ff(x)|\leq C\exp\{-\alpha |x|\},\quad x\in\R,\end{equation}
where $C=C(\mathcal{R})$ is a constant. Hence every $\Ff\in \Kk$ has exponential decay at $\pm\infty.$

Note that the classical hyperbolic secant generator belongs to $\Kk(1)$, since it is of the form \eqref{g}, where $\alpha=1$ and $\mathcal{R}(z)=z/(1+z^2).$

Observe that for every rational function $\Rr$ one may find  the largest  integer $k$ such that $$\Rr(e^{2\pi i /k}z)=c\mathcal{R}(z),\quad z\in\Cc,$$where $c$ is a  constant.
For example, 
\begin{itemize}
    \item if $\mathcal{R}(z)=z/(1+z^2)$, then  $k=2$ and $c=-1$;
    \item if $\mathcal{R}(z)=z/(1+z^4)$, then  $k=4$ and $c=i$;
    \item if $\mathcal{R}(z)=z/(1+z)^2$, then  $k=1$ and $c=1$. 
   \end{itemize}

Then, clearly, the generator $\Ff$ defined in \eqref{g} satisfies $\Ff(z+2\pi i/k\alpha)=c\,\Ff(z), z\in\Cc.$
   
\begin{definition}\label{def2}Assume $\Ff$ is defined in \eqref{g}.

$(i)$ We denote by $k(\Ff)$ is the greatest integer $k$ such that the equality
\begin{equation}\label{sym}\Ff(z+2\pi i/k\alpha)=c\Ff(z),\quad z\in\Cc,\end{equation}holds with some constant $c\in\Cc$. 

$(ii)$ We set $q(\Ff)=$\,{\rm deg}$\,Q$, where $Q$ is the denominator in \eqref{g}.
\end{definition}

\subsection{Class \texorpdfstring{$\kk(\alpha)$}{C(a)} }\label{sub_s_G}
\begin{definition}\label{gauss_def} 
We denote by $\kk(\alpha)$ the class of all  generators $\Ff$ defined by
\begin{equation}\label{g1}\Ff(z)=e^{-\alpha z^2/2}\mathcal{R}(e^{\alpha z})=e^{-\alpha z^2/2}\frac{P(e^{\alpha z})}{Q(e^{\alpha z})},\end{equation}
where $\alpha>0$,  $P,Q$ are non-trivial complex polynomials without  common zeros and $Q(x)\ne 0, x\geq0$.
Again, we set $q(\Ff)=${\rm deg}$\,Q$, where $Q$ is the denominator in \eqref{g1}.
\end{definition}

Clearly, the assumption $Q(x)\ne0,x\geq0,$ is necessary and sufficient for a generator $\Ff$ defined in \eqref{g1} to be bounded on $\R$. Moreover, every such $\Ff$ has "Gaussian decay" at $\pm\infty$.

    The case $P=Q\equiv 1$ corresponds to the classical case of Gaussian generator and has been studied previously in \cite{MR1188007} and \cite{MR1173118}. 
    The results below hold true for this particular case. 

Note that  all generators from $\Kk(\alpha)$ and $\kk(\alpha)$ defined in~\eqref{g} and~\eqref{g1}  belong to  the Wiener amalgam space $W_0$  defined in \eqref{wiener}.

\begin{remark}
    One may extend the definition of classes $\Kk(\alpha)$ and $\kk(\alpha)$ by considering complex parameter $\alpha$ satisfying $\rm{Re\,} \alpha > 0$ and prove similar results. 
\end{remark}

\subsection{Stability of \texorpdfstring{$\Gamma$}{Gamma}-shifts}\label{sub_i_stab} 
Given a generator $\Ff$, a basic property of quasi shift-invariant space $V^p_\Gamma(\Ff)$ is that the $\Gamma$-shifts of $\Ff$ are $l^p$-stable, i.e. there exist positive constants $C_1$ and $C_2$ such that
\begin{equation}\label{stab}C_1\|\cc\|_p \le \left\|\sum_{\gamma \in \Gamma} c_{\gamma} \Ff(\,\cdot - \gamma) \right\|_p \le C_2 \|\cc\|_p,\quad \mbox{for every } \cc=\{c_\gamma\} \in l^p(\Gamma).\end{equation}
 This property implies that $V^p_\Gamma(\Ff)$ is a closed subspace of $L^p(\R)$ and that the system $\{\Ff(\,\cdot - \gamma)\}_{\gamma \in \Gamma}$ forms an unconditional basis in this  space.

The stability property of $\Z$-shifts is well-studied. 
The following is an immediate corollary of Theorem~3.5  in  \cite{Jia1991}:

\begin{lemma}\label{integershifts}
    Assume $\Ff\in W_0$ and $p\in[1,\infty]$. Then the integer-shifts of $\Ff$ are $l^p$-stable if and only if the Fourier transform $\hat\Ff$ of $\Ff$ satisfies
    \begin{equation}\label{vanish}
      \hat\Ff \mbox{ does not vanish on any set } \Z+b, \ 0\leq b<1.  
    \end{equation} 
   \end{lemma}

Throughout the paper, we consider the standard form of Fourier transform,
$$\hat\Ff(t):=\int_\R e^{-2\pi i xt}\Ff(x)\,dx.$$

We also mention   paper \cite{MR3762092}, which proves the $l^2$-stability of   $\Gamma$-shifts for certain generators $\Ff$ under the condition that $\Gamma$ is a  complete interpolating sequence for  the  Paley--Wiener space  $PW_{\pi}^2$.

We would like to get conditions on $\Gamma$ and $\Ff$ that imply property \eqref{stab} for every value of $p\in[1,\infty].$ 
In fact, the right hand-side inequality in \eqref{stab} is true for  every separated set $\Gamma$, every generator $\Ff\in W_0$ and every $p\in[1,\infty]$, see Lemma \ref{LemmaA1} below. On the other hand, under a mild additional condition on the generator, it suffices to check that  the left hand-side inequality is true for   $p=\infty$:

\begin{theorem}\label{tstab}
    Let $\Gamma$ be a separated set and $\Ff\in W_0\cap C(\R)$. If the left hand-side inequality in \eqref{stab} is true for $p=\infty$, then it is true for  every $p\in[1,\infty]$.
\end{theorem}

In what follows, we will say that $\Gamma$-shifts of $\Ff$ are stable, if they are $l^p$-stable, for every $p\in[1,\infty].$

Let us now recall Beurling’s notion of the weak limit of a sequence of
sets. A sequence $\{\Gamma(n): n\in\N\}$ of separated subsets of $\R$ is said to converge weakly to a separated set
$\Gamma\subset\R$, if for every $R>0$ and $\epsilon > 0$,
there exists $n_{\epsilon,R}\in\N$ such that for all $n \geq n_{\epsilon,R}$ we have $$
\Gamma(n)\cap (-R,R) \subset \Gamma + (-\epsilon, \epsilon) \ \mbox{ and }\ \Gamma \cap(-R, R) \subset \Gamma(n) + (-\epsilon, \epsilon).$$

Given a separated and relatively dense set $\Gamma\subset\R$,  it is  easy to check that  every real sequence $\{s_j\}$ contains a subsequence $\{s_{j_n}\}$ such that the translates $\Gamma+s_{j_n}$ converge weakly to some  separated relatively dense set $\Gamma'$ as $n\to\infty$.
Let $W(\Gamma)$ denote the collection of all such weak limits. It is well-known (and easy to check) that  
\begin{equation}\label{upperlow}
    D^-(\Gamma')\geq D^-(\Gamma) \ \mbox{and }  D^+(\Gamma')\leq D^+(\Gamma),\quad \mbox{for every }\ \Gamma'\in W(\Gamma).
\end{equation}
As a simple example, we  observe that 
\begin{equation}\label{eex}
   \mbox{The set } W(\Z) \mbox{ consists of all sets } \Z+a,\ a\in [0,1). 
\end{equation}

We will now formulate some sufficient conditions for the stability of $\Gamma$-shifts for the  generators from the families $\Kk(\alpha)$ and $\kk(\alpha)$  defined above. These will be given in terms of the set  
 $\Gamma$ and the poles $w_j$ of the rational function $\mathcal{R}$ in \eqref{g} and \eqref{g1}.

Denote by $d(w_j)$ the order of the pole $w_j$ and set $d: = \max\limits_{j} d(b_j)$. Let  ${\rm pol}_d(\mathcal{R})$ denote the set of poles of $\mathcal{R}$ of order $d.$ 

We will consider the following assumptions:
 \begin{enumerate}
        \item[($\Xi'$)] The set ${\rm pol}_d(\mathcal{R})$ consists precisely of one element{\rm ;} 
        \item[($\Xi''$)] For every $\Gamma'\in W(\Gamma)$ there exists $w \in {\rm pol}_d(\mathcal{R})$ and $\gamma' \in \Gamma'$ such that
        \begin{equation}\label{elog}
         \log(w) - \log(w') \notin \alpha (\Gamma' - \gamma') \quad \text{for any } w' \in {\rm pol}_d(\mathcal{R}), \, w \neq w'.
        \end{equation}
    \end{enumerate}
      When $\Gamma = \Z$, it follows easily from  \eqref{eex} that  \eqref{elog} is equivalent to the simpler condition 
      \begin{enumerate}
          \item[($\Xi'''$)] $\log(w) - \log(w') \notin \alpha \Z$, \text{for every} $w' \in {\rm pol}_d(\mathcal{R}), \, w \neq w'.$  
                         \end{enumerate}

\begin{proposition}\label{pstab}
 {\rm (i)} If $\Ff\in\kk(\alpha)$ satisfies either $(\Xi')$ or $(\Xi''')$, then $\Z$-shifts of $\Ff$ are stable.

   {\rm (ii)} If $\Ff\in\Kk(\alpha)$ satisfies  either $(\Xi')$ or $(\Xi'')$, then $\Gamma$-shifts of $\Ff$ are stable.
\end{proposition}

\subsection{Sampling for generators from \texorpdfstring{$\Kk(\alpha)$}{K(a)}  and \texorpdfstring{$\kk(\alpha)$}{C(a)}}

Our first sampling theorem concerns the class of generators $\Kk(\alpha)$.

\begin{theorem}\label{sampling_thm1}
Given a generator $\Ff\in \Kk(\alpha), \alpha>0$, and two separated sets $\Lambda,\Gamma\subset \R$. Assume that  $\Gamma$-shifts of $\Ff$ are stable and that
\begin{equation}\label{d_pm_cond}
D^-(\Lambda)>\frac{q(\Ff)}{k(\Ff)}D^{+}(\Gamma).   
\end{equation}
Then $\Lambda$ is a sampling set for $V^p_\Gamma(\Ff)$, for every $p\in[1,\infty]$.
\end{theorem}

Recall that the numbers $q(\Ff)$ and $k(\Ff)$ are defined in Definition \ref{def2}.

\begin{remark}
    Note that condition~\eqref{d_pm_cond} is independent of parameter $\alpha$. However, above we assume the stability of $\Gamma$-shifts. This property can be violated for certain values of $\alpha$.
\end{remark}

Let
\begin{equation}\label{sec_eq1}
\Hh(x):=\frac{e^{x}}{e^{2x} + 1}\in \Kk(1)    
\end{equation}
denote the hyperbolic secant generator. Consider the family of all finite linear combinations
\begin{equation}\label{fff}\Ff(x)=\sum_{j=1}^N a_j \Hh(x-b_j)=\sum_{j=1}^Na_je^{b_j}\frac{e^x}{e^{2x}+e^{2b_j}}\in \Kk(1),\quad a_j\in\Cc, e^{2b_j}\not\in (-\infty,0).\end{equation}
Clearly, every $\Ff$ in \eqref{fff} satisfies $k(\Ff)=2$ and $q(\Ff)=2N.$ Assuming the stability of $\Z$-shifts of $\Ff$, Theorem \ref{sampling_thm1} implies that every separated set $\La$ satisfying $D^-(\Lambda)>N$ is a sampling set for $V_\Z^p(\Ff), 1\leq p\leq\infty.$ This result is sharp:

\begin{theorem}\label{sharp_thm}
  For every $N\geq2$ there exist  $a_j\in\R, b_j>0, j=1,\dots,N,$ and a separated set $\Lambda, D^{-}(\La)=N,$ such that   the generator $\Ff(z)$ in \eqref{fff} has  stable 
  $\Z$-shifts and $\La$ is not a uniqueness set for $V^{\infty}_{\Z}(\Ff)$. 
 \end{theorem}

Recall that a set $\Lambda$ is not a uniqueness set for $V^{\infty}_{\Z}(\Ff)$ if there is a non-trivial function 
$f\in V^{\infty}_{\Z}(\Ff)$ which vanishes on $\La$. Then, clearly, $\La$ is not a sampling set for $V^{\infty}_{\Z}(\Ff).$

For the generators $\Ff$ from the class $\kk(\alpha)$ we consider the integer shifts only. Our main sampling theorem for the shift-invariant space $V^2_{\Z}(\Ff)$ is as follows.
\begin{theorem}\label{sampling_thm2}
\begin{enumerate}
    \item[\rm{(i)}]  Assume a generator  $\Ff \in \kk(\alpha)$ has stable  $\Z$-shifts. If a separated set $\Lambda\subset\R$ satisfies
    \begin{equation}\label{e010}D^-(\Lambda)>q(\Ff)+1,\end{equation} then $\Lambda$ is a sampling set for $V^p_\Z(\Ff)$. 
    \item[\rm{(ii)}] For every $N\in\N$ there exist a generator $\Ff \in \kk(\alpha)$ with stable $\Z$-shifts and $q(\Ff)=N,$
    and a separated set $\Lambda$ satisfying $D^{-}(\Lambda) =q(\Ff)+1$, such that  $\Lambda$ is not a uniqueness set for $V^{\infty}_{\Z}(\Ff).$ 
\end{enumerate}
      \end{theorem}

\subsection{Interpolating sets for quasi shift-invariant spaces}

If for every $\cc=\{c_\gamma\} \in l^p(\Gamma)$ there is a function $f \in V_\Gamma^p(\Ff)$  such that $f(\lambda) = c_\lambda, \lambda\in\La$, then $\La$ is called  a set of interpolation  for $V_\Gamma^p(\Ff)$.

The duality between interpolation and sampling is well-known, see e.g. the discussion in~\cite{grs}. The following corollary follows from Theorem \ref{sampling_thm1}:

\begin{corollary}\label{interpol_cor}
Assume that $\Lambda,\Gamma$ and $\Ff$  satisfy assumptions of Theorem~\ref{sampling_thm1}.
Then $\Gamma$ is an interpolation set for $V^p_{\Lambda}(\Ff)$, for every $1\leq p\leq\infty.$
\end{corollary}

See the proof in Section \ref{s6}.
\subsection{Gabor frames}
Our results for the Gabor frames follow from the sampling theorems formulated above. 
We use the connection between Gabor frames and sampling theorems for shift-invariant spaces that previously turned out to be very fruitful, see e.g. \cite{grs, MR3053565}, and \cite{MR4782146} for a multi-dimensional setting.

Fix the standard notation for the operators of translation and modulation:
$$
M_{\xi} f = e^{2 \pi i \xi\, \cdot}f \quad \text{and} \quad T_{\lambda}f = f(\,\cdot - \lambda),\quad \text{where } f \in L^2(\R), \, (\lambda,\xi) \in \R^2.
$$

Let $\Lambda, \Psi$ be  separated real sets. For a generator $\mathcal{G}$ the {\it Gabor system } $G(\mathcal{G}, \Lambda\times \Psi)$ is the collection of all time-frequency shifts
\begin{equation}
    G(\Ff, \Lambda\times \Psi): = \{M_{\xi} T_{\lambda} \Ff,\,\, (\lambda,\xi) \in \La \times\Psi \}. 
\end{equation}

The system $G(\Ff, \Lambda\times \Psi)$ forms a {\it frame} in $L^2(\R)$ if there exist finite positive constants $A,B$ such that
$$
A\|f\|^2_{2} \leq \sum\limits_{(\lambda, \xi) \in \Lambda \times \Psi}|\langle f, M_{\xi}T_{\lambda} \mathcal{G} \rangle|^2 \leq B\|f\|^2_{2},
$$
for every $f \in L^2(\R).$

For the Gabor systems with generators from the families $\Kk(\alpha)$ and $\kk(\alpha)$,
we study the case of semi-regular lattices, i.e. $\Psi = \Z$. 
Our main result is as follows

\begin{theorem}\label{GF_cor}
 \begin{enumerate}
    \item[(i)] Assume   the Fourier transform $\hat\Ff$ of a generator $\Ff\in \Kk(\alpha),\alpha>0,$ satisfies  \eqref{vanish}.
    If a separated set $\La\subset\R$ satisfies   $D^{-}(\Lambda) > q(\Ff)/k(\Ff)$, 
   then   the system $G(\Ff, \Lambda \times  \Z)$ is a frame in $L^2(\R)$. 
  \item[(ii)] 
 There exist a  generator $\Ff\in \Kk(\alpha),\alpha>0,$ satisfying \eqref{vanish} and a separated set $\Lambda$ with the critical density $D^{-}(\Lambda)= q(\Ff)/k(\Ff)$ such that $G(\Ff, \Lambda \times  \Z)$ is not a frame in $L^2(\R)$. 
    \item[(iii)] Assume  a generator  $\Ff\in \kk(\alpha),\alpha>0,$  satisfies \eqref{vanish}.  If a separated set $\La$ satisfies     $ D^{-}(\Lambda) > q(\Ff)+1$, then   the system $G(\Ff, \Lambda \times \Z)$ is a frame in $L^2(\R)$. 

     \item[(iv)] 
      There exist a  generator $\Ff\in \kk(\alpha),\alpha>0,$ satisfying \eqref{vanish} and a separated set $\Lambda$ with the critical density $D^{-}(\Lambda)= q(\Ff)+1$ such that $G(\Ff, \Lambda \times  \Z)$ is not a frame in $L^2(\R)$. 
\end{enumerate}    
\end{theorem}

\subsection{Structure of the paper}
The paper is organized as follows. In Section~\ref{section_prelim} we fix  notations and formulate several known results. In Section~\ref{stab_section} we study stability of $\Gamma$-shifts and $\Z$-shifts for the generators from $\Kk(\alpha)$ and $\kk(\alpha). $
The uniqueness sets for the corresponding shift-invariant and quasi shift-invariant spaces are studied in Section~\ref{uniq_section}. In Section~\ref{non_uniq_sec} we present examples of  functions that vanish on certain sets of critical density. 
Combining results of Section~\ref{uniq_section} and Section~\ref{non_uniq_sec} and using the classical technique due to Beurling, we prove the sampling and interpolation theorems (Theorems~\ref{sampling_thm1}, \ref{sharp_thm}, \ref{sampling_thm2}, and Corollary~\ref{interpol_cor}) in Section~\ref{s6}. The results for Gabor frames are proved in Section~\ref{sec_GF}. Finally, in Section~\ref{sec_app}  we prove 
Theorem~\ref{tstab}. 

\section{Preliminaries}\label{section_prelim}

Throughout the paper, by $C$ we always denote positive constants. 
The notation ${\bf 1}_S$ stands for the indicator function of the set $S$.

Given a meromorphic  function $f(z), z\in\Cc$, denote by
$$
{\rm Zer}(f) = \left\{z \in \Cc : f(z) = 0 \right\}
$$ and Pol$\,(f)$ the {\it multisets}  of zeros and poles of $f$, i.e. each element is counted with its multiplicity. 
The notations ${\rm zer}(f)$ and ${\rm pol}(f)$ will stand for the set of zeros and poles of  $f$, i.e. the elements of  these sets  are pairwise distinct.

The next lemma is a simplified version of  Proposition~A.1 in \cite{MR3336091}, see also  Lemma~7.5 in \cite{grs}.

\begin{lemma}\label{interpol_sampl_lemma_A}
  Let $\Lambda$ and $\Gamma$ be separated sets in $\R$. Let $A \in \Cc^{\Lambda \times \Gamma}$ be a matrix such that
  $$
  |A_{\lambda, \gamma}| \le \theta(\lambda - \gamma) \quad \lambda \in \Lambda, \gamma \in \Gamma \quad \text{for some  } \theta \in W_0.
  $$
  Assume that there exist a $p_0 \in [1,\infty]$ and $C_0 > 0,$ such that
$$\|A\cc\|_{p_0} \geq  C_0\|\cc\|_{p_0}\quad \mbox{for all } \cc\in l^{p_0}(\Gamma).$$ 
Then there exists a constant $C > 0$ independent of $q$ such that, for all $q \in [1,\infty]$
$$\|A\cc\|_q \geq C\|\cc\|_q,\quad \mbox{for all } \cc\in l^q(\Gamma).$$ 
\end{lemma}

\begin{remark}[See Remark 8.2 in \cite{MR3336091}]\label{rr}
The constant $C$ in Lemma~{\rm\ref{interpol_sampl_lemma_A}}  depends only on the decay properties of the envelope $\theta$, the lower
bound for the given value of $p_0$, and on upper bounds for the relative separation of the
index sets.
 \end{remark}

To connect sampling
in shift-invariant spaces with Gabor frames,
we use the following lemma which is a particular case of a well-known result, see e.g. \cite[Theorem~2.3]{grs}. 

\begin{lemma}\label{GF_sampling_lemma}
Let $\Lambda \subset \R$ be a separated set,  and  $\Ff \in \kk(\alpha) \cup \Kk(\alpha)$. The following are equivalent:
    \begin{enumerate}
        \item[{\rm (i)}] The family $G(\Ff, -\Lambda\times \Z)$ is a frame for $L^2(\R)$.
        \item[{\rm (ii)}] For every $x \in [0,1)$ the set $\Lambda + x$ is a sampling set for $V^2_{\Z}(\Ff).$
    \end{enumerate}
\end{lemma}

\section{Stability of \texorpdfstring{$\Gamma$}{Gamma}-shifts}\label{stab_section}

In this section, we prove Proposition \ref{pstab} and present  examples of generators $\Ff$ from $\Kk(\alpha)$ and $\kk(\alpha)$ that do not have stable shifts.

Below we will need

\begin{lemma}\label{lstab}
  Given a generator $\Ff\in W_0$ and a separate set $\Gamma$. 
 If the left hand-side inequality in \eqref{stab} is not true, then 
  there is a set $\Gamma'\in W(\Gamma)$ and non-trivial coefficients  $\cc=\{c_{\gamma'}\}\in l^\infty(\Gamma')$ such that
  $$\sum_{\Gamma'} c_{\gamma'} \Ff(x-\gamma')\equiv 0.$$
\end{lemma}

\begin{remark}\label{rrr}
Observe without proof that the converse statement is also true: If the last equality holds for some   $\Gamma'\in W(\Gamma)$ and non-trivial $\{c_{\gamma'}\}$, then the  left hand-side inequality in \eqref{stab} is not true.   
\end{remark}

Observe that somewhat similar results are known, see e.g.  Theorem 2.1 (c) in \cite{grs}.
\begin{proof}
The proof below uses the (by now) standard Beurling's technique based on passing to a weak limit of translates of the set $\Gamma$.

 Write $\Gamma=\{...<\gamma_j<\gamma_{j+1}<...: j\in\Z\}$. Since  $\Gamma$-shifts of  $\Ff$ are not $l^\infty$-stable, for every $N\in\N$ there is a bounded sequence  $\cc(N)=\{c_j(N):j\in\Z\}$ such that
  $\|\cc(N)\|_{\infty} = 1$, and 
    $$
    \left\|\sum\limits_{j\in\Z} c_j(N) \Ff(x - \gamma_j)\right\|_{\infty} < 1 /N.
    $$
    
    Choose any  $k=k(N)\in\Z$ such that $|c_k(N)|>1/2$, and set $\Gamma(N):=\Gamma-\gamma_N,$ $d_j(N):=c_{k+j}(N)$. Then $0\in\Gamma(N)$ and $|d_0(N)|>1/2.$ Clearly, $\|\{d_j(N): j\in\Z\}\|_\infty=1$ and 
  $$
    \left\|\sum\limits_{j\in\Z} d_j(N) \Ff(x - \gamma_{k+j})\right\|_{\infty} < 1 /N.
    $$

    Now, passing to a subsequence, we may assume that 
$\Gamma(N)$ converges weakly to some {\it separated} set $ \Gamma':=\{\gamma_j':j\in\Z\}\in W(\Gamma)$ which contains the origin, and the sequence $\{d_j(N):j\in\Z\}$ converges for every $j$ to a bounded non-trivial sequence $\cc=\{c_j:j\in\Z\}$ as $N\to\infty.$
One can easily check that we have  
$$ \sum\limits_{j\in\Z} c_{j} \Ff(x - \gamma'_j)\equiv 0,$$ which proves the lemma.
   \end{proof}

\subsection{Stability of \texorpdfstring{$\Z$}{Z}-shifts for generators in \texorpdfstring{$\kk(\alpha)$}{C(a)}}\label{subs}
\begin{proof}[Proof of Proposition~\ref{pstab}, $(i)$]   
    The proof is by contradiction. We assume that one of the assumptions $(\Xi')$ or $(\Xi''')$ is satisfied, but the $\Z$-shifts of $\Ff\in\kk(\alpha)$ are not $l^\infty$-stable. Then, by Lemma \ref{lstab} and  (\ref{eex}), we find  a sequence $\cc=\{c_n\}\in l^\infty(\Z)$ such that 
    \begin{equation}\label{gw0}f(x):=\sum_{n\in\Z}c_n \Ff(x-n)=e^{-\alpha x^2/2}\sum_{n\in\Z}c_n e^{-\alpha n^2/2}e^{\alpha xn}\mathcal{R}(e^{\alpha(x-n)})\equiv 0.\end{equation}
    
Set $w=e^{\alpha x}.$ Then
\begin{equation}\label{gw1}
h(w):=\sum_{n\in\Z}c_ne^{-\alpha n^2/2}w^{n}\mathcal{R}(we^{-\alpha n})\equiv 0.
\end{equation}

Let  $w_1,...,w_m$ be the poles of $\mathcal{R}$, $d(w_j)$ the order of $w_j$ and   let ${\rm pol}_d(\mathcal{R})$ denote the set of poles of $\mathcal{R}$ of maximal order.
Clearly, $\mathcal{R}(we^{-\alpha n})$ has poles at $w_je^{\alpha n}, j=1,...,m$.
We may assume that $d(w_1)=d$. There are two possibilities:

First, assume that $(\Xi')$ is true, i.e. $d(w_j)<d, j=2,...,m$. Then clearly, the function $h$  has pole of order $d$ at each  point $w_1e^{\alpha n}, n\in\Z$, which contradicts \eqref{gw1}.

Second, assume that $(\Xi''')$ is satisfied:  there are $k$ poles whose order is equal to $d,$ say  $d(w_j)=d, j=1,...,k.$ If $\log(w_1/w_j)\not\in\alpha \Z,$ for all $2\leq j\leq k,$ then $w_1\ne w_je^{\alpha n}, n\in\Z,$ and so the function $h$ has a pole of order $d$ at $w_1$ (and also at every point $w_1 e^{\alpha n}, n\in\Z)$, which again contradicts \eqref{gw1}. This finishes the proof.
\end{proof}

We now present  an example of generator from $\kk(1)$ 
whose  $\Z$-shifts  are not stable.

\begin{ex}
Set
\begin{equation}\label{H-def}
    H(z):= U(z) e^{-z^2/2} := \left(A + \frac{e^{-1/2}}{e^{z-i} - 1} - \frac{e^i}{e^{z-1-i} -1} \right) e^{-z^2/2},
\end{equation}
where $A$ is a constant chosen such that 
\begin{equation}\label{H_zero}
    \int\limits_{-\infty}^{\infty} H(x) \, dx = 0.
\end{equation}    
\end{ex}

Clearly, $H(z)\in \kk(1)$ (see definition~\eqref{g1}) and  both assumption $(\Xi')$ and $(\Xi''')$ do not hold.

Let us show that the function $f\in V^\infty_\Z(H)$ defined by 
$$f(x):=\sum_{n\in\Z}H(x-n)$$vanishes identically. Since $f$ is $1$-periodic,
it suffices to prove that its Fourier coefficients vanish:

\begin{lemma}
    We have
    \begin{equation}
        \hat{H}(n) = 0,\quad n\in\Z. 
    \end{equation}
\end{lemma}
\begin{proof}
    Our goal is to show that
    $$
    0 = \hat{H}(n) = \int\limits_{-\infty}^{\infty}  U(x) e^{-x^2/2} e^{- 2 \pi i n x} \,dx = 0,\quad n\in\Z.
    $$
  Clearly,
  $$\hat H(n)=e^{-2\pi^2n^2}I_n, \quad I_n:=\int_{-\infty}^\infty e^{-(x+2\pi i n)^2/2}U(x)\,dx.$$
    Therefore,  it suffices to prove that  $I_n=0, n\in\Z.$

    The proof is by induction. 
    By \eqref{H_zero}, we have $I_0=0$.
    Assume that $I_n=0$ for $|n|\leq k, k\geq0$. We will show that $I_{k+1}=0$  (the proof of  $I_{-k-1}=0$ is similar). 
    
   Let us  integrate $e^{-(z+2\pi i k)^2/2} U(z)$ over the boundary of rectangle $\{z=x+iy: |x|\leq R, 0\leq y \leq 2\pi\}$ (integration is in positive direction with respect to the rectangle), and then  let $R \to \infty$. 
    It is clear that the integrals over the sides parallel to the imaginary axis tend to $0$ as $R\to \infty.$ The integral over the bottom side of the rectangle tends to $0$, since $I_k=0.$ 
    Since the function  $U(z)$ is $2\pi i$-periodic, the integral over the upper side tends to $-I_{k+1}$ as $R\to\infty$.
        Applying Cauchy's residue theorem, we obtain
$$-I_{k+1}=2\pi i  \left( \mathop{\mathrm{Res}}_{z= i} e^{-(z+2\pi i k)^2/2} U(z) +  \mathop{\mathrm{Res}}_{z= 1+ i}e^{-(z+2\pi i k)^2/2} U(z)\right). $$
Finally, from
    $${\mathrm{Res}}_{z= i} e^{-(z+2\pi i k)^2/2} U(z)=e^{-(i+2\pi ik)^2/2-1/2}=e^{2\pi k+2\pi^2k^2},$$
       $${\mathrm{Res}}_{z=1+ i} e^{-(z+2\pi i k)^2/2} U(z)=-e^{-(1+i+2\pi ik)^2/2+i}=-e^{2\pi k+2\pi^2k^2},$$
    we conclude that $I_{k+1}=0$, which  finishes the proof. \end{proof}

Observe that by Lemma \ref{integershifts},  $\Z$-shifts of the generator $H$ defined in \eqref{H-def} are not $l^p$-stable for every $p \in [1, \infty]$.

\subsection{Stability in \texorpdfstring{$V^p_{\Gamma}(g)$}{VpGamma(g)} }\label{subs_stab_g}

Let us now finish the proof of  Proposition \ref{pstab}.

\begin{proof}[Proof of Proposition~\ref{pstab}, $(ii)$]
    
    The proof is by contradiction and it is similar to the one above, and we will use the same notations.

    Assuming that $\Gamma$-shifts are not $l^{\infty}$-stable, by Lemma \ref{lstab} we get a set $\Gamma'\in W(\Gamma)$ and coefficients $\cc\in l^\infty(\Gamma')$ such that 
    \begin{equation}\label{hw2}
    h(w):=\sum_{\gamma' \in \Gamma'}c_{\gamma'} \mathcal{R}(we^{-\alpha \gamma'})\equiv 0.
    \end{equation}

Let $w_1$ be the pole of $\mathcal{R}$ of the maximal order $d$.

    If condition $(\Xi')$ is satisfied then $h$ has a pole at each point $w_1 e^{-\alpha \gamma'}, \gamma' \in \Gamma',
    $ which contradicts~\eqref{hw2}.

    Assume that $(\Xi'')$ is true: for any $w' \in {\rm pol}_d(\mathcal{R}) \setminus \{w_1\}$ and $\gamma' \in \Gamma'$ we have $w_1 e^{\alpha \gamma_0'} \neq w' e^{\alpha \gamma'}$. Therefore, the function $h$ has a pole of order $d$ at $w_1 e^{\alpha \gamma_0}$ which contradicts to \eqref{hw2}.
\end{proof}

For the class $\Kk(\alpha)$ we also provide an example of generator $H \in \Kk(1)$ that does not satisfy assumptions $(\Xi')$ and $(\Xi'')$, and whose $\Z$-shifts are not stable.

\begin{ex}\label{exg}
Set $$
    H(x) = \frac{e^x}{e^{2x} + 1} - \frac{e^{x+1}}{e^{2x}+e^{2}} \in \Kk(1).
$$  
\end{ex}

Since $H(x)=\Hh(x)-\Hh(x-1)\in\Kk(1),$ where $\Hh(x)$ is defined in \eqref{sec_eq1}, 
one may easily check  that the function
$\sum_{n\in\Z}H(x-n)$
belongs to $V_\Z^\infty(H)$ and vanishes identically on  $\R$. Moreover, since $\hat H(t)=(1-e^{2\pi i t})\hat\Hh(t)=0, t\in\Z,$
 Lemma \ref{integershifts} proves that  $\Z$-shifts of $H$ are not $l^p$-stable for every $p \in [1, \infty]$. 

We now formulate without proof the following
\begin{ex}
Choose any sequence $\delta_n\to0,n\to\pm\infty,$ satisfying $0<|\delta_n|<1/4,n\in\Z$. Set $\Gamma:=\{n+\delta_n:n\in\Z\}$. Then
 $\Gamma$-shifts of the generator $H$ in Example \ref{exg} are not $l^\infty$-stable.\end{ex}

One may prove the statement above using  e.g.  Remark \ref{rrr}.

\section{Uniqueness Sets}\label{uniq_section}

Following Beurling's approach (see \cite{MR10576141a}), to prove sampling theorems we first investigate the uniqueness sets for the spaces $V^\infty_{\Gamma}(\Ff)$.

\begin{proposition}\label{uniq_prop}
Given two separated sets $\Lambda,\Gamma\subset \R$.
\begin{enumerate}
\item[(I)] 
If  $\Lambda$ satisfies condition \eqref{d_pm_cond}, 
then it is a uniqueness set for $V^{\infty}_\Gamma(\Ff)$. 

\item[(II)]
If  $\Lambda$ satisfies condition \eqref{e010},  
  then it is a uniqueness set for $V^{\infty}_\Z(\Ff)$.
\end{enumerate}

\end{proposition}

\begin{remark}
The proof of Proposition~\ref{uniq_prop} is based on the classical Jensen formula. To apply it, we use the periodicity in the imaginary direction of the zeros of functions from the spaces $V^{\infty}_\Gamma(\Ff)$ for $\Ff \in \Kk(\alpha)$ and $V^{\infty}_\Z(\Ff)$ for $\Ff \in \kk(\alpha)$. This trick originates from the paper \cite{MR4047939}. 
\end{remark}

\subsection{Proof of Proposition~\ref{uniq_prop}, Part {\rm(I)}}\label{s41}
The proof below does not depend on the shape parameter $\alpha$. So, for simplicity throughout the proof we assume that $\alpha=1$. 

We must show that there is no non-trivial function
\begin{equation}\label{efunc}f(z)=\sum_{\gamma\in\Gamma}c_\gamma \mathcal{R}(e^{z-\gamma})\in V^\infty_\Gamma(\Ff),\quad \{c_\gamma\}\in l^\infty(\Gamma),\end{equation}
that vanishes on a set $\Lambda\subset\R$ satisfying  \eqref{d_pm_cond}.

The proof is by contradiction. Let us assume that such a function exists. Clearly, $f$ admits extension to the complex plane $\Cc$ as a meromorphic function satisfying $f(0)\ne\infty$. We may assume that $f(0)\ne0$ (otherwise we consider $g(z)=f(z-a)$ and 
$\Lambda' = \Lambda + a$ for a suitable $a\in\R.$) 

Recall that  Pol$(f)$ and Zer$(f)$  denote the multisets of poles and zeros of $f$, respectively. Recall also that we denote by $C$ different positive constants, and that the numbers $k(\Ff)$ and $q(\Ff)$ are defined in Definition \ref{def2}.

Set
\begin{equation}\label{eW_def}
W:=\log\left( \mbox{Pol}\,\mathcal{R}\right)\cap\{z=x+iy\in\Cc: 0< y<2\pi\}.  
\end{equation}
This means that  $\mathcal{R}(z)=\infty$ if and only if $ z=e^{u+iv}, u+iv\in W,$ and the number of occurrences of $z$ in $e^W$ is equal to the multiplicity of the pole of $\mathcal{R}$ at $z$. 
 Since $f$ is $2\pi i$-periodic, we have$$\mbox{Pol}(f)\subseteq\{z\in\Cc: e^{z-\gamma}\in e^W,\gamma\in\Gamma \}=\Gamma+W+2\pi i \Z.$$ 

Recall that $f$ vanishes on $\Lambda$. It follows from \eqref{sym} that $f$ also vanishes on the set $\Lambda+\{2\pi i j/k(\Ff): j=0,...,k(\Ff)-1\}=\Lambda+(2\pi i)/k(\Ff)\cdot  Z$, where $ Z:=
\{0,...,k(\Ff)-1\}.$ 
Therefore, 
$$\mbox{Zer}(f)\supseteq \Lambda+\frac{2\pi i}{k(\Ff)} Z+2\pi i \Z.$$

The proof below is based on the following two lemmas:

\begin{lemma}\label{l01}
 We have $$n_0(t)\geq n_\infty(t)+Ct^2,$$  for some $C>0$ and for every sufficiently large $t$. 
\end{lemma}
Above $n_0(t)$ and $n_\infty(t)$ denote the number of zeros and poles (counting multiplicities) of $f$ in the circle  $B_t:=\{z\in\Cc: |z|\leq t\}$, respectively.

\begin{lemma}\label{l02}
 There is a sequence $R_j\to\infty$ such that  
 $$|f(z)| \leq \|\cc\|_{\infty}\sum_{\gamma\in\Gamma}\left| \mathcal R(e^{z-\gamma})\right|\leq C R_j^{q(\Ff)+1},\quad |z|=R_j,\ j\in\N.$$
\end{lemma}

Let us now check that the lemmas above  contradict to the classical Jensen formula for meromorphic functions (see e.g. \cite{MR1400006}, Ch. 2.4) 
    \begin{equation}\label{jen}\int_0^R\frac{n_0(t)-n_\infty(t)}{t}dt=\frac{1}{2\pi}\int_0^{2\pi}\log|f(Re^{i\theta})|d\theta-\log|f(0)|,\end{equation}where we assume that $f(z)\ne0,\infty$ on the circle $|z|=R$.

Indeed, by Lemma \ref{l01}, the left hand-side of the formula is larger than  $CR^2$ as $R\to\infty,$ while by Lemma \ref{l02}, the right-hand size has a logarithmic growth on a sequence $R=R_j\to\infty$. Therefore, to finish the proof of proposition it remains to prove these lemmas.

\begin{proof}[Proof of Lemma \ref{l01}]

We start with three claims:

Let us denote by $\re\, W$ the multiset $\re\,W:=\{\re\, z: z\in W\}$. Observe that \begin{equation}\label{rew}
    \#\re\,W=\# W=q(\Ff).
\end{equation}

Denote by $\tilde\La$ the multiset of points belonging to $\La$, where each point of $\tilde\Lambda$ has multiplicity $k(\Ff)$. We use the definition (\ref{lowed}) to define the lower density $D^-(\tilde\Lambda)$.
\begin{claim}\label{c1} We have $$D^+(\Gamma+\re\,W)=q(\Ff)D^+(\Gamma), \quad D^-(\tilde \Lambda)=k(\Ff)D^-(\Lambda).$$    
\end{claim} 

\begin{claim}\label{c2}
    Let numbers $R>0$ and $a,b\in\Cc$ satisfy
    $$|\re\, a|\leq|\re \,b|<R.$$
Then $$\#\left((a+2\pi i \Z)\cap B_R\right)\geq \#\left((b+2\pi i \Z)\cap B_R\right)-1.$$\end{claim}

\begin{claim}\label{c3}
    There is a positive number $\rho$ such that for all open intervals $I,J\subset\R$ of length $|I|=|J|=\rho$  we have
    $$\#\left(\tilde \Lambda\cap I\right)=k(\Ff)\#\left(\Lambda\cap I\right)\geq \# \left((\Gamma+\re\,W)\cap J\right)+k(\Ff).$$
\end{claim}

    We omit the simple proofs of Claims \ref{c1} and \ref{c2}. The last claim is a simple consequence of Claim \ref{c1}, \eqref{rew} and \eqref{d_pm_cond}.

    Now, choose a number $r>0$ such that $|\gamma +\re\,e^w|\geq  \rho$,  for every $\gamma\in\Gamma, |\gamma|\geq r$ and  every $w\in W$, where $\rho$ is the number in Claim \ref{c3}. Set $\Gamma':=\Gamma\setminus (-r,r).$
By Claim \ref{c3} we can find a subset $\Lambda'\subset\Lambda$ such that there is a bijection 
$$T:\Lambda'+\frac{2\pi i}{k(\Ff)} Z\to \Gamma'+ W$$ satisfying $|\re\, T(u)|\geq |\re\, u|$ and such that \begin{equation}\label{e00}\#\left((\Lambda\setminus\Lambda')\cap (j\rho,(j+1)\rho)\right)\geq 1,\quad j\in\Z. 
\end{equation}    
Since the set $\Gamma\setminus\Gamma'$ is finite, it follows that $$\#\left(((\Gamma\setminus\Gamma')+W+2\pi i \Z)\cap B_R\right)\leq CR.$$Hence, by Claim \ref{c2} we get the estimate
$$\#\left((\Lambda'+\frac{2\pi i}{k(\Ff)} Z+2\pi i \Z)\cap B_R\right)\geq \#\left((\Gamma+W+2\pi i \Z)\cap B_R\right)-CR. $$On the other hand, by \eqref{e00}  for all large enough $R$ one  gets the estimate  
$$\#\left((\Lambda+\frac{2\pi i}{k(\Ff)} Z+2\pi i \Z)\cap B_R\right)\geq \#\left((\Lambda'+\frac{2\pi i}{k(\Ff)} Z+2\pi i \Z)\cap B_R\right)+CR^2,$$
where the constant $C$ depends on $\rho.$ This finishes the proof of Lemma \ref{l01}.
\end{proof}

\begin{proof}[Proof of Lemma \ref{l02}] Recall that $\mathcal{R}=P/Q$ satisfies conditions (A) - (C) in  Definition~\ref{admis_def}. It follows that there exist constants $C,L>0$ such that
\begin{equation}\label{e03}|\mathcal{R}(z)|\leq C|z|,\quad |z|\leq e^{-L} \ \  \mbox{and }  \ |\mathcal{R}(z)|\leq \frac{C}{|z|},\quad |z|\geq e^{L}.\end{equation}One can also easily check that
\begin{equation}\label{e0}|\mathcal{R}(z)|\leq\frac{C}{(\mbox{dist}(z,e^W))^{q(\Ff)}},\quad e^{-L}<|z|<e^L.\end{equation}
where $e^W$ defined in \eqref{eW_def}.

 Let $f$ be defined in \eqref{efunc}.  It is easy to see that there is a constant $C$ such that $$\#\left(\mbox{Pol}(f)\cap B_R\right)\leq CR^2,\quad R\geq 1.$$Therefore, there is a sequence $R=R_j\to\infty$ such that \begin{equation}\label{e04}\mbox{dist}(\mbox{Pol}(f),\partial B_R)\geq \frac{C}{R},\quad \partial B_R:=\{z\in\Cc: |z|=R\}.\end{equation}

We now fix such a number $R$ and split $\Gamma$ into two sets: $$\Gamma_1:=\{\gamma\in\Gamma: |\gamma|\geq R+L\},\,\, \Gamma_2:=\Gamma\cap (-R-L,R+L).$$ By \eqref{e03}, $$|\mathcal{R}(e^{z-\gamma})|\leq Ce^{-|x-\gamma|},\quad \gamma\in\Gamma_1, \ z=x+iy, \ |z|=R. $$ Since $\Gamma$ is a separated set, this gives
$$\sup_{|z|=R}\sum_{\gamma\in\Gamma_1}\left|c_\gamma \mathcal{R}(e^{z-\gamma})\right|\leq C\|\cc\|_\infty\sup_{x\in\R}\sum_{\gamma\in\Gamma_1}e^{-|x-\gamma|}\leq C,$$where the last constant does not depend on $R$.

Further,  by \eqref{e03} and \eqref{e0}, for every $z,|z|=R$ and $\gamma\in\Gamma_2,$
$$|\mathcal{R}(e^{z-\gamma})|\leq \sup_{|z|=R,w\in W}\frac{C}{|e^{z-\gamma}-e^w|^{q(\Ff)}}+C\leq \sup_{|z|=R,w\in W}\frac{C}{|e^{z-w-\gamma}-1|^{q(\Ff)}}.$$

We now use a simple inequality $$|e^\zeta-1|\geq C\delta,\quad \mbox{dist}(\zeta, 2\pi i\Z)\geq \delta.$$
   It is clear that $\# \Gamma_2\leq CR.$ Hence, the last inequality and  \eqref{e04} imply
   \begin{equation}\label{esimi}\sum_{\gamma\in\Gamma_2}\left|c_\gamma \mathcal{R}(e^{z-\gamma})\right|\leq C \|\cc\|_\infty\#\Gamma_2R^{q(\Ff)} =C R^{q(\Ff)+1}.\end{equation}
This finishes the proof of Lemma~\ref{l02}.
\end{proof}

\medskip

\subsection{Proof of Proposition~\ref{uniq_prop}, Part {\rm(II)}}

The proof is   similar to the proof of Part (I). 
However, recall that in this case, we do not exclude the option ${\rm deg}\,P \geq {\rm deg}\,Q.$

For simplicity, we assume that the shape parameter $\alpha=1$. The proof in the general case is similar.

We argue by contradiction and assume that there is a non-trivial function $$f(z)=\sum_{n\in\Z}c_n \Ff(z-n)=\sum_{n\in\Z}c_n e^{-(z-n)^2/2}\mathcal{R}(e^{z-n})
\in V^\infty_\Z(\Ff),\ \cc=\{c_n\}\in l^\infty(\Z),$$which vanishes on $\Lambda$. We have to show that this implies a contradiction.

As in Lemma \ref{l02} above, let $n_0(t)$ and $n_\infty(t)$ denote the number of zeros and poles of $f$ in $\{z\in\Cc:|z|=t\}$, respectively.

\begin{lemma}\label{l47}  We have

(i) $n_0(t)\geq n_\infty (t)+(1+C)t^2/2$, for some $C>0$ and all large enough $t.$

(ii) There is sequence $R=R_j\to\infty$ such that $$\log|f(z)|\leq y^2/2 +C\log R,\quad z=x+iy, |z|=R.$$
    
\end{lemma}

\begin{proof}
Instead of conditions \eqref{e03}  and \eqref{e0}, we now have the conditions
$$|\mathcal R(z)|\leq C(1+|z|^k),\quad |z|<e^{-L}, \ |z|>e^L$$ and 
$$|\mathcal{R}(z)|\leq\frac{C}{(\mbox{dist}(z,e^W))^{q(\Ff)}},\quad e^{-L}<|z|<e^L,$$for some $k\in\N$ and $C,L>0.$

Let  $R_j\to\infty$ be a sequence satisfying  condition \eqref{e04}.  Fix an element $R=R_j$ and set $Z_1:=\{n\in\Z: |n|\geq R+L\}$ and $Z_2:=\Z\cap (-R-L,R+L)$. Using the first inequality above, we get
$$\sum_{n\in Z_1}\left|c_n e^{-(z-n)^2/2}\mathcal R(e^{z-n})\right|\leq Ce^{y^2/2}\sum_{n\in Z_1}e^{-(x-n)^2/2}\left(1+e^{k(x-n)}\right)\leq Ce^{(y^2+k^2)/2}=Ce^{y^2/2}.$$
Next, by the second inequality above, similarly to  \eqref{esimi}, we get 
$$\sum_{n\in Z_2}\left|c_n e^{-(z-n)^2/2}\mathcal R(e^{z-n})\right|\leq C\|\cc\|_\infty e^{y^2/2}R^{q(\Ff)}\sum_{n\in Z_2}e^{-(x-n)^2/2}
\leq CR^{q(\Ff)+1}e^{y^2/2}.$$

We conclude that for every $f\in V_\Z^\infty(\Ff)$ we have
\begin{equation}\label{e05}|f(x+iy)|\leq CR^{q(\Ff)+1}e^{y^2/2}, \quad |z|=R_j,\end{equation} where $C$ depends only on $f.$
This proves part (ii) of Lemma \ref{l47}.

Given $\Ff\in\kk(1)$ and $f \in V_{\Z}^{\infty}(\Ff)$, observe that
$$f(z+2\pi i )=\sum_{n\in\Z}c_ne^{-(z+2\pi i -n)^2/2}\mathcal R(e^{z+2\pi i -n})=e^{-2\pi i z +2\pi^2}f(z).$$ Hence, since $f(\lambda)=0,\lambda\in\La,$ then $f$ vanishes on the set $\La+2\pi i \Z.$
Therefore, similarly to  Section \ref{s41}, we have $$\mbox{Pol}(f)\subseteq \Z+W+2\pi i\Z,\quad \mbox{Zer}(f)\supseteq \La+2\pi i \Z.$$

By estimate \eqref{e010}, we may split $\La$ into two sets, $\La=\La_1\cup\La_2$ satisfying $D^-(\La_1)>q(\Ff)$ and $D^-(\La_2)>1.$ Then  $$n_0(t)\geq n_1(t)+n_2(t):=\# (\La_1+2\pi i \Z)\cap B_t+\# (\La_2+2\pi i \Z)\cap B_t.$$ 
As in the proof of Lemma \ref{l01}, we have 
$$n_1(t)\geq n_\infty(t)+ Ct^2,\quad \mbox{for some } C>0.$$

To prove part (i) of Lemma \ref{l47}, it suffices to prove  the following 
\begin{claim}\label{cla}
We have   $$n_2(t)\geq t^2/2 ,\quad \mbox{for all large enough } t.$$
\end{claim}

Given a convex set $I\subset\Cc, 0\in I,$ and a discrete set of points $\Delta\subset\Cc$, consider the  density of $\Delta$ defined as $$d(I,\Delta):=\lim\inf_{r\to\infty}\frac{\#\Delta\cap(rI)}{r^2|I|},$$where $|I|$ denotes the 2D-measure (area) of $I$. If $d(I,\Delta)>0$, then for every $\epsilon>0$ we have \begin{equation}\label{edelta}\#\Delta\cap rI\geq (1-\epsilon)r^2|I|d(I,\Delta),\end{equation}
for all sufficiently large $r$.

Assume that $I$ is the square $$I=\{z=x+iy\in\Cc: \max\{|x|,|y|\leq 1\}\}.$$ It is easy to check that 
$$d(I, \La_2+2\pi i \Z)\geq D^-(\La_2)/2\pi.$$
However, it is well-knwon that the density $d(I,\Delta)$ does not depend on the choice of $I$, see e.g. \cite[Lemma~4]{la67}.
Since $D^-(\La_2)>1,$ Claim \ref{cla} follows from \eqref{edelta}. \end{proof}

Again, to arrive at contradiction, we use Jensen's formula~\eqref{jen} for meromorphic functions. From Lemma~\ref{l47}, part (ii), we see that the right hand-side of \eqref{jen} is bounded above by
$$
\frac{1}{2\pi}\int_0^{2\pi}\log|f(Re^{i\theta})|d\theta \le \frac{R^2}{4\pi}\int_0^{2\pi} \sin^2(\theta) d \theta + C \log R  = \frac{R^2}{4} + C \log R, R=R_j\to\infty.
$$
On the other hand, part (i) of Lemma~\ref{l47} shows that the left hand-side of \eqref{jen} is larger than $(1+C)R^2/4, C>0,$ for all large enough $R$, which is a contradiction. This finishes the proof of Proposition~\ref{uniq_prop}.

\section{Non-uniqueness Sets}\label{non_uniq_sec}

In this section, we study zero sets of functions from shift-invariant spaces $V^{\infty}_{\Z}(\Ff)$ generated by $\Ff \in \kk(\alpha)$ or $\Ff \in \Kk(\alpha)$. More precisely, we build functions from these shift-invariant spaces that vanish on sets of critical density.

Let us start with 
\begin{proof}[Proof  of Theorem~\ref{sharp_thm}]

Given $N \in \N$, $\alpha>0$, and $0< b_1 < \dots < b_N$ such that
\begin{equation}\label{poles_indep}
     b_j - b_k \notin \Z, \quad \text{for every  } \,\, j \neq k.  
\end{equation}
It suffices to find a separated set $\Lambda$ satisfying $D^{-}(\Lambda) = N$, coefficients $\{a_j\}_{j=1}^N \subset \R$, and a non-trivial function $f$ from $V^{\infty}_{\Z}(\Ff)$ such that $f$ vanishes on $\Lambda$, where 
$$
\Ff(x) = \sum\limits_{j=1}^N \frac{a_j e^{\alpha (x+b_j)}}{e^{2 \alpha x} + e^{2 \alpha b_j}}.
$$

We will distinguish the cases where $N$ is an even or an odd positive integer.

\medskip\noindent {\bf Case 1.} Assume that $N\in 2\N.$ Consider the functions
$$\varphi_j(x)=\sum_{n\in\Z}\frac{e^{\alpha (x+b_j-n)}}{e^{2\alpha (x-n)}+e^{2 \alpha b_j}},\quad j=1,...,N.$$Clearly, these functions are $1$-periodic. Observe that they are also linearly independent, since it follows from~\eqref{poles_indep} that they have different poles. Therefore, we can find $N$ points $0<x_1<...<x_N<1$ such that the system of $N$ equations
\begin{equation}\label{equ}\sum_{j=1}^N a_j\varphi_j(x_l)=(-1)^l,\quad l=1,...,N,\end{equation}has a
real solution $a_1,...,a_N.$ This means that the function
\begin{equation}\label{ef}f(x):=\sum_{j=1}^N a_j \varphi_j(x)\in V_{\Z}^\infty(\Ff)\end{equation}has at least $N-1$ sign changes on the interval $(0,1).$ However, since $f$ is also $1$-periodic, it either vanishes at $0$ or has an even number of sign changes on $(0,1)$, and so $f$ has at least $N$ distinct zeros on $[0,1).$ We see that ${\rm Zer}\,(f)$ contains a $1$-periodic  set of density $N$.
This finishes the proof for the case $N\in 2\N$.

\begin{figure}[!ht]
\centering
\includegraphics[height=8cm]{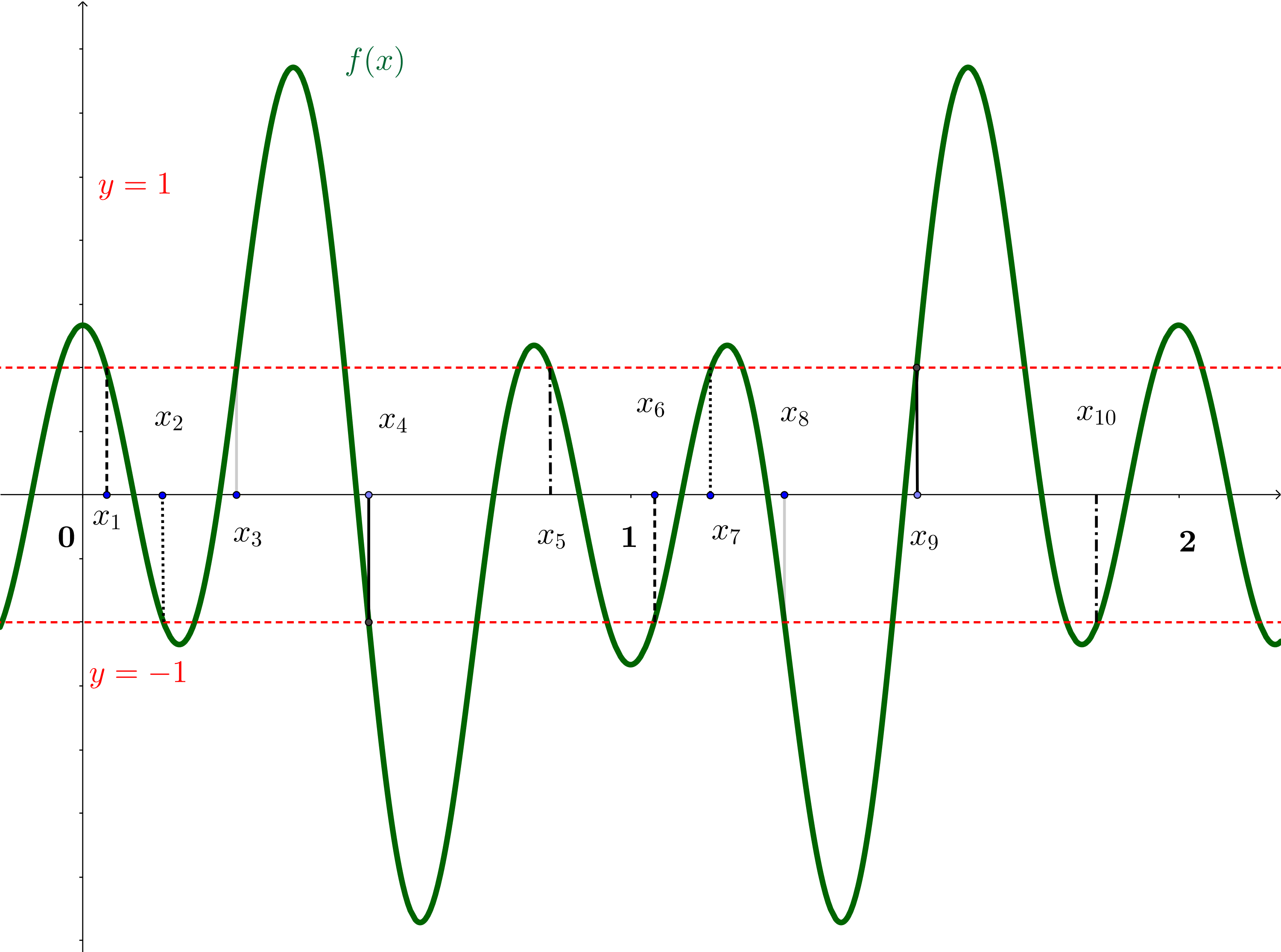}
\caption{Case 2 for $N=5$}
\label{img1}
\end{figure}

\medskip\noindent {\bf Case 2.} Assume that $N\in 2\N+1.$ Consider the functions
$$\varphi_j(x)=\sum_{n\in\Z}(-1)^n\frac{e^{\alpha (x+b_j-n)}}{e^{2 \alpha (x-n)}+e^{2b_j}},\quad j=1,...,N.$$ 
Note that $\{\varphi_j\}_{j=1}^N$ are linearly independent,  $2$-periodic, and $\varphi_j(x+1)=-\varphi_j(x)$.

Similarly to Case 1, we define  $f$ by
\begin{equation}\label{ef_psi}f(x):=\sum_{j=1}^N a_j \varphi_j(x) \end{equation}
and  find $N$ points $0<x_1<...<x_N<1$ such that \eqref{equ} has a real 
solution $a_1,...,a_N.$ Clearly, $f \in V_{\Z}^{\infty}(\Ff).$

Set $x_{N+j}:=1+x_j,j=1,...,N.$ Since $N\in 2\N+1$ and $f(x_{N+j})=f(1+x_j)=-f(x_j)$, we see that $f$ must have at least one sign change (and therefore a zero) on each interval $(x_j,x_{j+1}), j=1,..., 2N-1,$ see Figure~\ref{img1}. This means that $f$ has at least $2N-1$ sign changes  on $(x_1,x_{2N})$. However, since $f$ is $2$-periodic, it either vanishes at $0$ or  has an even number of sign changes on $[0,2).$ Therefore, it has at least $2N$ different zeros on $[0,2).$ Since $f$ is $2$-periodic, we see that ${\rm Zer}\,(f)$ contains a set of density $N$. 
This finishes the proof of Theorem~\ref{sharp_thm}.    
\end{proof}

Next, we prove a similar statement for the  generators of the form
  \begin{equation}\label{fgene}\Ff(x)= e^{-\alpha x^2/2} \left( a_0 + \sum_{j=1}^N \frac{a_j}{e^{\alpha x}+e^{\alpha b_j}} \right),\quad a_0,a_1,...,a_N\in\R.\end{equation}
Clearly, $\Ff\in\kk(\alpha)$. 

\begin{proposition}\label{prop_uG}
    Given $N \in \N, \alpha > 0$, and $0 < b_1 < \dots < b_N,$ such that    
    \begin{equation}\label{bj_independent_2}
         b_j - b_k \notin \Z \quad \text{for every  } j,k = 1, \dots, N, j \neq k.  
    \end{equation}
       There exist a separated set $\Lambda, D^{-}(\Lambda) = N+1$, coefficients $a_0,...,a_N\in \R$, and a non-trivial function $f\in V^{\infty}_{\Z}(\Ff),$ where $\Ff$ is defined in \eqref{fgene}, such that $f$ vanishes on $\Lambda$.
\end{proposition} 

\begin{proof}
    The argument follows the proof of Theorem~\ref{sharp_thm}. Again, we consider  the cases  $N$ is  even and  odd integer separately.
    
    Let us  assume that $N \in 2\N$ and sketch the proof leaving the details to the reader. 
    The proof of the second  case is also left to the reader.
    Set
    $$
    \psi_0(x) = \sum_{n \in \Z}(-1)^n e^{-\alpha (x-n)^2/2}, \quad
    \psi_j(x):= \sum\limits_{n \in \Z} (-1)^n \frac{e^{-\alpha (x-n)^2/2}}{e^{\alpha (x-n)}+e^{\alpha b_j}}, \,\, j =1, \dots, N.
    $$ Note that for every $j$ the function $\psi_j$ is $2$-periodic and $\psi_j(x+1) = - \psi_j(x).$
    Using the linear independence of the system $\{\psi_i\}_0^N$ that follows from \eqref{bj_independent_2}, we can find $N+1$ points  $0<x_0 < \dots < x_{N} < 1,$ such that 
    the system of $N+1$ equations 
    $$
    \sum\limits_{j=0}^N a_j \psi_j(x_l) = (-1)^l, \quad l = 0,\dots, N.
    $$
    has a real solution $a_0, \dots, a_N.$
    Therefore, the function 
    $$
    f(x) = \sum\limits_{j=0}^{N} a_j \psi_j(x), \quad f\in V^{\infty}_{\Z}(\Ff),
    $$
    has the same alternating properties as the function $f$ defined in~\eqref{ef_psi}. The rest of the proof is  similar to the proof of  Case 2 above. 
\end{proof}

\section{Proofs of Sampling and Interpolation Theorems}\label{s6}

\subsection{Proof of Theorem~\ref{sampling_thm1}}

We split the proof into two steps.

The proof does not depend on $\alpha$, so we set $\alpha = 1$.

{\noindent\bf Step 1.} We start with proving that $\Lambda$ is a sampling set for $V^{\infty}_{\Gamma}(\Ff).$ 
The proof is by contradiction.

 Let us assume that condition (\ref{d_pm_cond}) is satisfied and that $\Lambda$ is not a sampling set for $V^{\infty}_{\Gamma}(\Ff).$ 
Hence, for every integer $n$ there exist functions such that
$$f_n(x) = \sum\limits_{\gamma \in \Gamma} c^{(n)}_{\gamma} \Ff(x-\gamma), \quad f_n \in V^{\infty}_{\Gamma}(\Ff), \quad \|\cc^{(n)}\|_{\infty} = 1,$$ and $\gamma_n \in \Gamma$ such that
$|c^{(n)}_{\gamma_n}| > 1/2$ and
$$
\|f\big|_{\Lambda}\|_{\infty} = \sup\limits_{\lambda \in \La} \left| \sum\limits_{\gamma \in \Gamma} c^{(n)}_{\gamma} \Ff(\lambda - \gamma)  \right| < \frac{1}{n}.
$$

Using Beurling's technique, similarly to the proof of Lemma \ref{lstab}, one can find 
non-empty sets  $\Lambda'\in W(\Lambda)$, $\Gamma'\in W(\Gamma)$ and 
non-trivial coefficients ${\bf d}=\{d_{\gamma'}\}\in l^\infty(\Gamma')$ such that the function
$$h(x):=\sum_{\gamma'\in\Gamma'}d_{\gamma'}\Ff(x-\gamma')$$vanishes on the set $\Lambda'.$
Using~\eqref{upperlow} and Proposition~\ref{uniq_prop}, we arrive at a contradiction.
Therefore $\Lambda$ is a sampling set for $V^{\infty}_{\Gamma}(\Ff).$

{\noindent \bf Step 2.}
We show that if $\Lambda$ is a sampling set for $V^{\infty}_{\Gamma}(\Ff)$ then $\Lambda$ is a sampling set for $V^{p}_{\Gamma}(\Ff)$ for any $1 \le p < \infty.$ To this end, we use the approach developed in \cite{MR3336091}. 
Consider an operator
$A:l^p(\Gamma) \to l^{p}(\Lambda)$ given by
\begin{equation}\label{ea}A\cc:= \left\{\sum\limits_{\gamma \in \Gamma}c_{\gamma}\Ff(\lambda - \gamma):\lambda\in\Lambda\right\},  \quad \cc=
\{c_{\gamma}\}.\end{equation}
This operator is given by the matrix $A = \{\Ff(\lambda - \gamma)\} \in \Cc^{\Lambda \times \Gamma}.$ 

Since $\Lambda$ is a sampling set for $V^{\infty}_{\Gamma}(\Ff)$, the operator $A$ is bounded from below for $p = \infty$:
$$
\sup\limits_{\lambda \in \Lambda}\left|\sum\limits_{\gamma \in \Gamma} c_{\gamma} \Ff(\lambda-\gamma)\right| \geq C\|\cc\|_\infty. 
$$
By Lemma~\ref{interpol_sampl_lemma_A}, we deduce that $A$ is bounded from below in $l^p$ for any $1 \le p < \infty.$ Therefore, for every function 
$$
f(x) = \sum\limits_{\gamma \in \Gamma} c_{\gamma} \Ff(x - \gamma)\in V^p_{\Gamma}(\Ff)
$$we get 
$$
\sum\limits_{\lambda \in \Lambda}|f(\lambda)|^p = 
\|A\cc\|_p^p\geq C\|\cc\|_p^p \geq C \|f\|^p_{p},
$$where the last inequality follows from Lemma \ref{lp-stable-bessel}.
This completes the proof. 

\subsection{Proof of Theorem~\ref{sampling_thm2}} The proof of part 
 $(i)$ is similar to the proof of  Theorem~\ref{sampling_thm1} above.

Part (ii) follows from Proposition~\ref{prop_uG}.

\subsection{Proof of Corollary~\ref{interpol_cor}}We will use the following result from Banach theory:
{\it Let $X$ and $Y$ be Banach spaces. 
Let
$U : X\to Y$ be a bounded operator.

$(i)$  $U$ is onto if and only if   the dual operator $U^\ast:Y^\ast\to X^\ast $ is bounded from below. 

$(ii)$  $U$ is bounded from below if and only if   $U^\ast$ is onto.} 

 For the statement (i) we refer the reader to \cite[Theorem~E9]{MR0262773}. The statement (ii) is also well-known, see \cite[Section 10.2.4, Exercise 3]{Kadets2018}.

Let $A:l^p(\Gamma) \to l^{p}(\Lambda)$ be the operator defined in \eqref{ea} 
 The dual operator $A^\ast :l^{p'}(\Lambda) \to l^{p'}(\Gamma), 1/p+1/p'=1,$ is given by
$$
A^\ast{\bf d}:= \left\{\sum\limits_{\lambda \in \Lambda}d_{\lambda} \Ff(\lambda - \gamma):\gamma\in\Gamma\right\},\quad {\bf d}=\{d_\lambda\}.
$$Since $\Lambda$ is a sampling  set for every space $V^p_\Gamma(\Ff),1\leq p\leq\infty$, it  is bounded from below.
Using the above  theorem, one may conclude that $A^\ast$ is onto for every $1\leq p'\leq\infty$. This means  that $\Gamma$ is an interpolation set for $V^{p'}_{\Lambda}(\Ff).$

\section{On Gabor Frames}\label{sec_GF}

\begin{proof}[Proof of Theorem~\ref{GF_cor}]
The proof is a simple consequence of Lemma \ref{GF_sampling_lemma} and our sampling Theorems~\ref{sampling_thm1} and \ref{sampling_thm2}.

(i) Let $\Ff\in \Kk(\alpha)$ satisfy the assumptions of Corollary~\ref{GF_cor} (i).
   Then, by  Lemma \ref{integershifts},   $\Z$-shifts of $\Ff$ are stable. 
Also,  since  $D^{-}(-\Lambda + x)=D^-(\Lambda)$, it follows from the assumptions,   Theorem~\ref{sampling_thm1} (with $\Gamma = \Z$) and Lemma~\ref{GF_sampling_lemma} that $G(\Ff, \Lambda \times  \Z)$ is a frame in $L^2(\R)$.

 The proof of part (iii) is  similar to the proof of part (i).

Parts (ii) and (iv) follow from Lemma \ref{integershifts} and Theorems \ref{sharp_thm},  \ref{sampling_thm2}.
 
\end{proof}

\section{Stability of \texorpdfstring{$\Gamma$}{Gamma}-shifts}\label{sec_app}

\subsection{Estimate from above}\label{SubSA1}

\begin{lemma}\label{LemmaA1}
 Let $1 \le p \le \infty$, $\Gamma$ be a separated set and $\Ff\in W_0$. 
 Then 
    \begin{equation}\label{lp-stable-bessel}
        \left\| \sum\limits_{\gamma \in \Gamma} c_{\gamma} \Ff(\cdot - \gamma) \right\|_p \leq N^{1/q}(\Gamma)\|\Ff\|_{W_0} \|\cc\|_{p},\quad \mbox{for every }
        \cc\in l^p(\Gamma),
    \end{equation}  where   $q=p/(p-1)$ and $N(\Gamma)$ is the covering constant
    $$N(\Gamma):=\sup_{x\in\R}\sum_\Gamma{\bf 1}_{[0, 1]}(x+\gamma).$$
\end{lemma}

See Theorem 2.1 in \cite{Jia1991} for a slightly more general result for the multi-dimensional integer-shifts.

\begin{proof}
Since the set $\Gamma$ is separated, we have   $N(\Gamma)<\infty.$ 

The proof is obvious when $p=\infty.$ Therefore, we assume that $1\leq p<\infty$.

Clearly, for every $h\in L^q(\R)$ and $k\in\Z$, we have
$$\sum_{\gamma\in\Gamma}\int_{k+\gamma}^{k+1+\gamma}|h(x)|^q\,dx\leq N(\Gamma)\|h\|_q^q.$$
We will use  this inequality at the end of the following calculations:
$$\int_\R\left| \sum_\Gamma c_\gamma \Ff(x-\gamma)h(x)\right|\,dx\leq \sum_{k\in\Z}\sum_\Gamma |c_\gamma| \int_k^{k+1}|\Ff(x)h(x+\gamma)|\,dx\leq$$
$$\sum_{k\in\Z}\sum_\Gamma\|\Ff\|_{L^\infty(k,k+1)}|c_\gamma|\int_{k+\gamma}^{k+1+\gamma}|h(x)|dx\leq $$$$\sum_{k\in\Z}\|\Ff\|_{L^\infty(k,k+1)}\sum_\Gamma|c_\gamma|\left(\int_{k+\gamma}^{k+1+\gamma}|h(x)|^qdx\right)^{1/q}\leq 
N(\Gamma)^{1/q}\|\Ff\|_{W_0}\|\cc\|_{l^p}\|h\|_q.$$

Finally,
$$\left\|\sum_\Gamma c_\gamma \Ff(x-\gamma)\right\|_p=\sup_{0<\|h\|_q<\infty}\frac{1}{\|h\|_q}\int_\R \left|\sum_\Gamma c_\gamma \Ff(x-\gamma)h(x)\right|dx\leq N(\Gamma)^{1/q} \|\Ff\|_{W_0}\|\cc\|_{l^p}.$$

\end{proof}

\subsection{Estimate from below}\label{SubsA2}

Here we prove Theorem \ref{tstab}.

We assume that $\Gamma$-shifts of $\Ff$ are $l^\infty$-stable, and prove that they are $l^p$-stable, for every $p\in[1,\infty].$
This latter means that $\|f\|_p\geq K_p\|\cc\|_p$, for every function $f$ of the form
 $$f(x) = \sum\limits_{\gamma} c_{\gamma} \Ff(x - \gamma),\quad \cc\in l^p(\Gamma).$$

Choose any positive number $\delta$ and denote by $\ccc$ the family of all sequences 
$$\La=\{\lambda_k\}_{k \in \Z}=\{...<\lambda_k<\lambda_{k+1}<...: \lambda_k\in [(k+1/4)\delta, (k+3/4)\delta],k\in\Z\}.$$

For every $\Lambda\in\ccc$ we denote by $A_\La$ the discretization operator 
defined by
\begin{equation}
    \{A_\La\cc\} :=\{f(\lambda):\lambda\in\La\}= \left\{\sum\limits_{\gamma \in \Gamma} c_{\gamma} \Ff(\lambda - \gamma):\lambda\in\La\right\},\quad \cc=\{c_\gamma\}.
\end{equation}

\begin{claim}\label{discr_claim}
    There exists $\delta_0>0$ such that $\|A_{\La}\cc\|_\infty>K\|\cc\|_\infty,$ for some $K>0$ and every $\cc\in l^\infty(\Gamma)$,  $\La\in\ccc$ and  $\delta\leq\delta_0$.
\end{claim}

\begin{proof}[Proof of  Claim~{\rm\ref{discr_claim}}]
Since $\Gamma$-shifts of $\Ff$ are $l^\infty$-stable, we have $\|f\|_\infty\geq K_\infty\|\cc\|_\infty$, for  all $\cc\in l^\infty(\Gamma)$ and some $K_\infty>0.$ Then for every $f\in V^\infty_\Gamma(\Ff)$ we may find a point $x_0=x_0(f)$ such that
$|f(x_0)|\geq (K_\infty/2)\|\cc\|_\infty.$ 

Let $\Lambda\in\ccc$. Clearly, there exists  $\lambda_0\in\Lambda$ such that $|\lambda_0-x_0|<2\delta.$
This gives $$|f(\lambda_0)-f(x_0)|\leq \|\cc\|_\infty\sum_\gamma |\Ff(\lambda_0-\gamma)-\Ff(x_0-\gamma)|=\sum_{\gamma:|\gamma-x_0|\leq R}+\sum_{\gamma:|\gamma-x_0|>R},$$where 
by \eqref{wiener} one may choose $R$  so large that 
$$\sum_{\gamma:|\gamma-x_0|> R}|\Ff(\lambda_0-\gamma)|+|\Ff(x_0-\gamma)|<K_\infty/4.$$

On the other hand, since $\Ff\in C(\R)$, it is uniformly continuous on $[-R,R]$. Therefore,
$$\sum_{\gamma:|\gamma-x_0|\leq R}|\Ff(\lambda_0-\gamma)-\Ff(x_0-\gamma)|<K_\infty/4,$$provided $\delta$ is sufficiently small, which proves the claim.
\end{proof}
In what follows we assume that $\delta\leq\delta_0$.
Hence, by the definition of the family $\ccc$, Lemma~\ref{interpol_sampl_lemma_A} and Remark \ref{rr}, we see that there is a constant $C$ such that 
\begin{equation}\label{final}
 \|A_\La\cc\|_p>C\|\cc\|_p, \mbox{ for every }\cc\in l^p(\Gamma) \mbox{ and every } \La\in\ccc.\end{equation}

 Choose any $p,1\leq p<\infty$. Given any $f\in V^p_\Gamma(\Ff)$, choose points $\lambda_k\in [(k+1/4)\delta, (k+3/4)\delta]$ so that 
 $|f(x)|\geq |f(\lambda_k)|, x\in [(k+1/4)\delta, (k+3/4)\delta]$. Then 
 $$\La:=\{\lambda_k:k\in\Z\}\in\ccc.$$Using \eqref{final},  we get 
$$\|f\|_p^p\geq (\delta/2)\sum_{k \in \Z} |f(\lambda_k)|^p= (\delta/2) \|A_\La\cc\|_p^p\geq (\delta/2) C^p\|\cc\|_p^p,$$which proves Theorem \ref{tstab}. 

\section{Acknowledgements}
The second author is grateful to K.~Gr\"{o}chenig, M.~Faulhuber, and I.~Shafkulovska for fruitful discussions. The authors are grateful to the anonymous referee for several helpful suggestions.

\section{Declarations}
{\bf Competing Interests:} On behalf of all authors, the corresponding author states that there is no conflict of interest. 

{\bf Data availability.} There is no data associated with this article.

\end{document}